\documentclass[a4paper,reqno,12pt]{amsart}
\usepackage{amsmath,amstext,amssymb,amsopn,amsthm,mathrsfs}
\usepackage{xcolor}
\numberwithin{equation}{section}
\allowdisplaybreaks

\newtheorem{theorem}{Theorem}[section]
\newtheorem{propo}[theorem]{Proposition}
\newtheorem{lemma}[theorem]{Lemma}
\newtheorem{corollary}[theorem]{Corollary}

\newcommand{\N}{\mathbb{N}}
\newcommand{\R}{\mathbb{R}^d}

\newcommand{\Om}{\Omega}
\newcommand{\EE}{\mathcal{E}}
\newcommand{\OO}{\mathcal{O}}

\newcommand{\D}{{\rm Dom}}

\newcommand{\pr}{\mathbf{P}}
\newcommand{\ex}{\mathbf{E}}
\newcommand{\set}{\Omega}
\newcommand{\g}{g}
\newcommand{\gD}{g^{\set}}

\def\a{\alpha}

\def\la{\lambda}
\def\v{\varphi}
\def\s{\sigma}

\begin{document}
\title[Reflection principles]{Reflection principles for functions of Neumann and Dirichlet Laplacians on open reflection invariant subsets of $\R$}
%%%%%%%%%%%%%%%%%%%%%%%%%%%%%%%%%%%%%%%%%%%%%%
\author[J. Ma\l{}ecki]{Jacek Ma\l{}ecki}
\address{Jacek Ma\l{}ecki \endgraf\vskip -0.1cm
         Wydzia\l{} Matematyki \endgraf\vskip -0.1cm
         Politechnika Wroc\l{}awska \endgraf\vskip -0.1cm
         Wyb{.} Wyspia\'nskiego 27 \endgraf\vskip -0.1cm
         50-370 Wroc\l{}aw, Poland \endgraf \vskip -0.1cm
         }
\email{Jacek.Malecki@pwr.edu.pl}

\author[K. Stempak]{Krzysztof Stempak}
\address{Krzysztof Stempak \endgraf\vskip -0.1cm
         Wydzia\l{} Matematyki \endgraf\vskip -0.1cm
         Politechnika Wroc\l{}awska \endgraf\vskip -0.1cm
         Wyb{.} Wyspia\'nskiego 27 \endgraf\vskip -0.1cm
         50-370 Wroc\l{}aw, Poland \endgraf \vskip -0.1cm
         }
\email{Krzysztof.Stempak@pwr.edu.pl}

\thanks{Research  supported by funds of Faculty of Pure and Applied Mathematics, Wroc\l{}aw University of Science and Technology, 
\# 0401/0121/17.}

%%%%%%%%%%%%%%%%%%%%%%%%%%%%%%%%%%%%%%%%%%%%%%
\begin{abstract} For an open subset $\Omega$ of $\mathbb R^d$, symmetric with respect to a hyperplane and with positive part $\Omega_+$, we consider the Neumann/Dirichlet Laplacians $-\Delta_{N/D,\Omega}$ and  
$-\Delta_{N/D,\Omega_+}$. Given a Borel function $\Phi$ on $[0,\infty)$ we apply the spectral functional calculus and consider the pairs 
of operators $\Phi(-\Delta_{N,\Omega})$ and $\Phi(-\Delta_{N,\Omega_+})$, or  $\Phi(-\Delta_{D,\Omega})$ and $\Phi(-\Delta_{D,\Omega_+})$. We prove relations between the integral kernels 
for the operators in these pairs, which in particular cases of $\Omega_+=\mathbb{R}^{d-1}\times(0,\infty)$ and $\Phi_{t}(u)=\exp(-tu)$, $u \geq 0$, $t>0$, were known as reflection principles 
for the Neumann/Dirichlet heat kernels. These relations are then generalized to the context of symmetry with respect to a finite number of mutually orthogonal hyperplanes.
 
%35K08 PDE - Parabolic equations and systems - heat kernel
%47B25 Operator theory - Symmetric and selfadjoint operators (unbounded)
%60J65 Probability theory and stochastic processes - Markov processes - Brownian motion.
\end{abstract}

\maketitle
%%%%%%%%%%%%%%%%%%%%%%%%%%%%%%%%%%%%%%%%%%%%%%
%%%%%%%%%%%%%%%%%%%%%%%%%%%%%%%%%%%%%%%%%%%%%%
\section{Introduction} \label{sec:intro}
Let $\Om$ be a nonempty open subset of $\mathbb R^d$,  $d\geq 1$, and let $\Delta=\sum_1^d \partial_j^2$ denote the Laplacian. If not otherwise stated, $-\Delta_\Om$ will mean the differential operator
$f\mapsto -\Delta f$ with domain $C^\infty_c(\Om)$ (the space of compactly supported $C^\infty$ functions on $\Om$), which is dense in $L^2(\Om)$. Clearly  $-\Delta_\Om$ is symmetric,
$$
\langle(-\Delta_\Om)f,g\rangle_{L^2(\Om)}=\langle f,(-\Delta_\Om)g\rangle_{L^2(\Om)}, \qquad f,g\in {\rm Dom}(-\Delta_\Om)=C^\infty_c(\Om),
$$
and non-negative, $\langle (-\Delta_\Om)f,f\rangle_{L^2(\Om)}\ge0$ for $f\in {\rm Dom}(-\Delta_\Om)$. 

The Sobolev spaces $H^n(\Om)$ and $H^n_0(\Om)$,
$n\in \N$, denoted also as $W^{n,2}(\Om)$ and $W^{n,2}_0(\Om)$,  are defined as follows (see, for instance, \cite[Appendix D]{Sch} or \cite[Chapter 6]{D2}): $H^n(\Om)$ is the linear space of 
functions $f\in L^2(\Om)$ for which the distributional derivative $\partial^\a f$
belongs to $L^2(\Om)$ for all $\a\in\N^d$, $|\a|\le n$, endowed with the inner product
$$
\langle f,g\rangle_{H^n(\Om)}=\sum_{|\a|\le n}\langle \partial^\a f,\partial^\a g\rangle_{L^2(\Om)},
$$
and  $H^n_0(\Om)$ is the closure of  $C^\infty_c(\Om)$ in $(H^n(\Om), \|\cdot\|_{H^n(\Om)})$. Then  $H^n(\Om)$ (and thus also $H^n_0(\Om)$) is a Hilbert space.

Let $\mathfrak{t}_\Om$ be the sesquilinear form defined on the domain $H^1(\Omega)$ by
$$
\mathfrak{t}_\Om[f,g]=\int_{\Om}(\nabla f)(x)\cdot\overline{(\nabla g)(x)}\,dx=\int_{\Om}\sum_{j=1}^d \partial_jf(x)\,\overline{\partial_jg(x)}\,dx. 
$$
The \textit{Neumann Laplacian} on $\Om$, denoted by $-\Delta_{N,\Om}$, is defined as the operator on $L^2(\Omega)$ associated with the  
form $\mathfrak{t}_{N,\Omega}:=\mathfrak{t}_\Om$; in particular, ${\rm Dom}(-\Delta_{N,\Om})\subset{\rm Dom}(\mathfrak{t}_{N,\Omega}):= H^1(\Omega)$. On the other hand, 
the \textit{Dirichlet Laplacian} on $\Om$, denoted $-\Delta_{D,\Om}$, is defined as the operator on $L^2(\Omega)$ associated with the 
form $\mathfrak{t}_{D,\Omega}$, which is the restriction of $\mathfrak{t}_{\Omega}$ to $H^1_0(\Om)$; in particular, ${\rm Dom}(-\Delta_{D,\Om})\subset{\rm Dom}(\mathfrak{t}_{D,\Omega}):= H^1_0(\Omega)$.
Since the forms $\mathfrak{t}_{N,\Omega}$ and $\mathfrak{t}_{D,\Omega}$ are Hermitian, closed and non-negative, the associated operators are self-adjoint and non-negative. 
See \cite[Chapter 10 and Section 3 of Chapter 12]{Sch}. Each of the operators $-\Delta_{N/D,\Om}$ is indeed an extension of $-\Delta_{\Om}$; this follows from the definitions in terms of forms, with an application of Green's formulas for functions from Sobolev classes, that can be found, for instance, in \cite[Appendix D]{Sch}.
We also mention that $-\Delta_{D,\Om}$ coincides with the Friedrichs extension 
of $\overline{-\Delta_{\Om}}$, the closure of $-\Delta_{\Om}$. See \cite[Section 10.6.1]{Sch}.

In the setting of a general open set $\Om$ it is known (see, for instance, \cite[Section 10.6.1]{Sch}) that
\begin{equation*}%\label{dom}
{\rm Dom}(-\Delta_{D,\Om,})=H^\Delta(\Om)\cap H^1_0(\Om)
\end{equation*}
and
\begin{equation*}%\label{act}
-\Delta_{D,\Om}f=-\Delta f\,\,\,\,{\rm for}\,\,\, f\in {\rm Dom}(-\Delta_{D,\Om}).
\end{equation*}
Here, for the sake of convenience, we used the notation
$$
H^\Delta(\Om)=\{f\in L^2(\Om)\colon \Delta f\in L^2(\Om)\}
$$
and for $f\in L^2(\Om)\subset \mathcal C^\infty_c(\Om)'$, $\Delta f$ is understood  in the distributional sense.
Note that $H^2(\Om)\subset H^\Delta(\Om)$ but in general the inclusion may be proper. Contrary to the case of the Dirichlet Laplacian much less is known about the explicit description of 
${\rm Dom}(-\Delta_{N,\Om})$, the domain of the Neumann Laplacian, in the setting of general $\Om\subset \R$.

If $\Om$ is an open bounded subset  in $\R$, $d\ge2$, with boundary $\partial\Om$ of class $C^2$, or an open bounded subset of $\mathbb{R}$, then there are much finer results concerning properties of $-\Delta_{D,\Omega}$ and $-\Delta_{N,\Omega}$. In particular,
in this case the Dirichlet Laplacian refers to vanishing boundary values at $\partial\Om$ and the Neumann Laplacian refers to vanishing directional normal derivatives at $\partial\Om$.  See, for instance,
 \cite[Theorems 10.19 and 10.20]{Sch}.

The case $\Om=\R$ is special. Then (see, for instance, \cite[Theorem D.3, Appendix D]{Sch}), for any $n\in\N$ we have
$$
H^n(\R)=H^n_0(\R)=\{f\in L^2(\R)\colon \|\cdot\|^n\mathcal{F}f\in L^2(\R)\},
$$
and the latter space, for $n=2$ coincides with $H^\Delta(\R)$
(here $\mathcal{F}$ denotes the Fourier-Plancherel transform on $L^2(\R)$). Hence, 
$$
{\rm Dom}(-\Delta_{D,\R})=H^2(\R) \,\,\,{\rm and}\quad -\Delta_{D,\R}f=-\Delta f\,\,\,\,{\rm for}\,\,\, f\in H^2(\R).
$$
Since $H^1(\R)= H^1_0(\R)$, from the very definitions of the considered operators, it follows that $-\Delta_{N,\R}=-\Delta_{D,\R}$.

By the spectral theorem, we associate with the Dirichlet Laplacian $-\Delta_{D,\Om}$ the semigroup  $\{\exp(-t(-\Delta_{D,\Om}))\}_{t>0}$ of bounded on $L^2(\Om)$ operators, called the \textit{Dirichlet heat semigroup}.
Each  $\exp(-t(-\Delta_{D,\Om}))$, $t>0$, is an integral operator with a kernel $p_t^{D,\Om}(x,y)$, that is for every $f\in L^2(\Om)$ there holds
$$
\exp(-t(-\Delta_{D,\Om}))f(x)=\int_{\Om}p_t^{D,\Om}(x,y)f(y)dy, \qquad x-a.e.
$$
Moreover, as a function on $(0,\infty)\times\Om\times\Om$, $p_t^{D,\Om}(x,y)$ is $C^\infty$ and strictly positive. See \cite[Theorem 5.2.1]{D1}.
Then $\{p_t^{D,\Om}(x,y)\}_{t>0}$,  is called the \textit{Dirichlet heat kernel on} $\Om$.

Analogously, we consider the \textit{Neumann heat semigroup} $\{\exp(-t(-\Delta_{N,\Om}))\}_{t>0}$ associated with  $-\Delta_{N,\Om}$. As before,
each  $\exp(-t(-\Delta_{N,\Om}))$, $t>0$, is an integral operator with a kernel $p_t^{N,\Om}(x,y)$
which, as a function on $(0,\infty)\times\Om\times\Om$, is $C^\infty$ and strictly positive. %See [?, ?].
Then $\{p_t^{N,\Om}(x,y)\}_{t>0}$,  is called the \textit{Neumann heat kernel on} $\Om$.

Clearly, in the special case of $\Om=\R$, skipping in the notation the symbol $\R$, we have
$$
p_t^N(x,y)=p_t^D(x,y)=p_t(x,y):=(4\pi t)^{-d\slash2}\exp(-\|x-y\|^2\slash4t).
$$
It is also known that for the half-space $\mathbb{R}^{d}_+:=\mathbb R^{d-1}\times (0,\infty)$ the corresponding Neumann/Dirichlet heat kernels, denoted $p_t^{N,+}(x,y)$ and $p_t^{D,+}(x,y)$, are related to $p_t(x,y)$ by
$$
p_t^{N,+}(x,y)=p_t(x,y)+p_t(\tilde x,y),\qquad p_t^{D,+}(x,y)=p_t(x,y)-p_t(\tilde x,y), 
$$
where $x,y\in \mathbb{R}^{d}_+$, and $\tilde x=(\check x,-x_d)$ denotes the reflection point of $x=(\check x,x_d)\in \mathbb{R}^{d}_+$,  with respect to the 
hyperplane orthogonal to $(0,\ldots,0,1)$.

The aim of this paper is to prove that a similar principle holds for kernels of operators emerging in spectral calculus applied to Neumann/Dirichlet Laplacians in the general setting of an open $\Om\subset\R$, which is symmetric with respect to the hyperplane $\langle v\rangle^\bot$. 

Let $\s_v$ be the orthogonal reflection with respect to  $\langle v\rangle^\bot$ (see Section \ref{sec:prel} for details). Recall that $-\Delta_{N/D,\Om}$ and $-\Delta_{N/D,\Om_+}$ are non-negative and hence their spectra are contained in $[0,\infty)$.

%%%%%%%%%%%%%%%%%%%%%%%%%%%%%%%%%%%%%%%%%%%
\begin{theorem} \label{thm:main} Let $\Om$ be an open subset of $\R$ symmetric with respect to $\langle v\rangle^\bot$  with $\Om_+$ as its positive part. 
Let $\Phi$ be a Borel function on $[0,\infty)$. 
Assume that $\Phi(-\Delta_{N,\Om})$ is an integral operator with the kernel $K^\Phi_{-\Delta_{N,\Om}}$. Then $\Phi(-\Delta_{N,\Om_+})$ is also an integral operator with the kernel 
$K^\Phi_{-\Delta_{N,\Om_+}}$ given by
\begin{equation}\label{kerN}
K^\Phi_{-\Delta_{N,\Om_+}}(x,y)=K^\Phi_{-\Delta_{N,\Om}}(x,y)+ K^\Phi_{-\Delta_{N,\Om}}(\s_v(x),y),\qquad x,y\in\Om_+.
\end{equation}
Similarly, if $\Phi(-\Delta_{D,\Om})$ is an integral operator with the kernel $K^\Phi_{-\Delta_{D,\Om}}$, then $\Phi(-\Delta_{D,\Om_+})$ is also an integral operator with the kernel 
$K^\Phi_{-\Delta_{D,\Om_+}}$ given by
\begin{equation}\label{kerD}
K^\Phi_{-\Delta_{D,\Om_+}}(x,y)=K^\Phi_{-\Delta_{D,\Om}}(x,y)- K^\Phi_{-\Delta_{D,\Om}}(\s_v(x),y),\qquad x,y\in\Om_+.
\end{equation}
\end{theorem}
%%%%%%%%%%%%%%%%%%%%%%%%%%%%%%%%%%%%%%%%%%%
As a direct corollary of Theorem \ref{thm:main} we obtain the following identities that can be called the \textit{reflection principles} for the Neumann and Dirichlet heat kernels.
%%%%%%%%%%%%%%%%%%%%%%%%%%%%%%%%%%%%%%%%%%%
\begin{corollary} \label{cor:one}
Let $\Om$ be an open subset of $\R$ symmetric with respect to $\langle v\rangle^\bot$  and let  $p_t^{N,\Om_+}$ and $p_t^{D,\Om_+}$, and $p_t^{N,\Om}$ and $p_t^{D,\Om}$, denote the Neumann and the Dirichlet heat kernels on $\Om_+$ and $\Om$, respectively.
Then
\begin{equation}\label{ker1}
p_t^{N,\Om_+}(x,y)=p_t^{N,\Om}(x,y)+p_t^{N,\Om}(\s_v(x),y) ,\qquad x,y\in\Om_+,\quad t>0,
\end{equation}
and
\begin{equation}\label{ker2}
p_t^{D,\Om_+}(x,y)=p_t^{D,\Om}(x,y)-p_t^{D,\Om}(\s_v(x),y) ,\qquad x,y\in\Om_+,\quad t>0.
\end{equation}
\end{corollary}
%%%%%%%%%%%%%%%%%%%%%%%%%%%%%%%%%%%%%%%%%%%
The paper is organized as follows. Section \ref{sec:prel} is devoted to the statements and proofs of auxiliary results and the proof of Theorem \ref{thm:main}.
In Section \ref{sec:prob} we use a probabilistic approach to verify \eqref{ker2}. Finally, in Section \ref{sec:ex} we first show how to extend Theorem \ref{thm:main}
to a more complex setting of multiple reflections associated to an orthogonal root system. Then we discuss several applications of Theorem \ref{thm:main}
by considering resolvents, Riesz potential operators and heat semigroups associated to the Neumann/Dirichlet Laplacians on open sets in $\R$ that result in reflection
principle formulas for the corresponding integral kernels. These include resolvent kernels and thus also Green's functions, Riesz potential kernels, and heat kernels.
We also discuss concrete examples. In particular, we recover \eqref{ker1} and \eqref{ker2} for several symmetric open sets $\Om$
by comparing formulas for the Neumann/Dirichlet heat kernels for $-\Delta_{N/D,\Om}$ and $-\Delta_{N/D,\Om_+}$ which are known to be given in terms of series.

%%%%%%%%%%%%%%%%%%%%%%%%%%%%%%%%%%%%%%%%%%%%%%%%%%%%%%%%%%%%%%%%%%%%%%%%%%
\section{Preliminaries and proofs of main results} \label{sec:prel}
%%%%%%%%%%%%%%%%%%%%%%%%%%%%%%%%%%%%%%%%%%%%%

Given a vector $0\neq v\in\R$ let $\s_v$ denote the orthogonal reflection with respect to  the hyperplane $\langle v\rangle^\bot$ perpendicular to $v$,
$$
\s_v(x)=x-2\frac{\langle v,x\rangle}{\|v\|^2} v,  \qquad x\in\R.
$$
If $d=1$, then the "hyperplane" reduces to $\{0\}$ and $\s_v(x)=-x$.

Let $\Om$ be an open set in $\R$ symmetric in $\langle v\rangle^\bot$, that is $\s_v(\Om)=\Om$. We distinguish the positive part of $\Om$ by setting
$$
\Om_+=\{\omega\in\Om\colon \langle\omega,v\rangle>0\}.
$$
Given a function $f$ on $\Om_+$ we define $\EE f$ and $\OO f$, its even and odd extensions on $\Om$ with respect to $\langle v\rangle^\bot$, 
by setting for $x\in \Om_-:=\s_v(\Om_+)$,
$$
\EE f(x)=f(\s_v(x))\qquad {\rm and} \qquad \OO f(x)=-f(\s_v(x)).
$$
On the set $\Om\cap \langle v\rangle^\bot$ of Lebesgue measure zero, the definitions of both extensions are immaterial but, if necessary, for instance for $\OO f$, we can set  
$\OO f(x)=0$ for $x\in\Om\cap \langle v\rangle^\bot$. Also, for a function $F$ on $\Om$, by $F_{\rm even}$ and $F_{\rm odd}$ 
we mean the even and odd parts of $F$ (with respect to $\langle v\rangle^\bot$), 
$$F_{\rm even}(x)=(F(x)+F(\s_v(x))/2\qquad {\rm and} \qquad  F_{\rm odd}(x)=(F(x)-F(\s_v(x))/2;
$$
if not otherwise stated, we consider $F_{\rm even}$ and $F_{\rm odd}$ as restrictions to $\Om_+$, hence treat them as functions on $\Om_+$. 

In what follows we shall use, without further mentioning, the following identities,
\begin{equation*}%\label{aux}
\int_{\Om}\EE \phi\cdot \Psi=2\int_{\Om_+}\phi\cdot \Psi_{\rm even},\qquad \int_{\Om}\OO \phi\cdot \Psi=2\int_{\Om_+}\phi\cdot \Psi_{\rm odd};
\end{equation*}
here $\phi$ and $\Psi$ are suitable functions on $\Om_+$ and $\Om$, respectively. Also, if $V$ is a linear space of functions on $\Om$, then by $V_{\rm even}$ and 
$V_{\rm odd}$ we denote the linear space of functions on $\Om_+$ consisting of even and odd parts of 
functions from $V$, respectively.

%%%%%%%%%%%%%%%%%%%%%%%%%%%%%%%%%%%%%%%%%%%
\begin{lemma} \label{lem:first}
We have
\begin{equation}\label{H1}
H^1(\Om_+)=\big(H^1(\Om)\big)_{\rm even}\qquad and \quad H^1_0(\Om_+)=\big(H^1_0(\Om)\big)_{\rm odd}.
\end{equation}
\end{lemma}
%%%%%%%%%%%%%%%%%%%%%%%%%%%%%%%%%%%%%%%%%%%
\begin{proof}
In the case $d=1$, both identities in \eqref{H1} immediately follow from known characterizations of $H^1(I)$ and $H^1_0(I)$, where $I\subset\mathbb R$ is an open interval 
(see, for instance, \cite[Appendix E]{Sch}). Thus, we can assume that $d\ge2$.

Since the Laplacian is rotationally invariant, in what follows without any loss of generality, but only for the sake of simplicity, we can assume 
(and we do this in the proof of Proposition \ref{pro:one}) that $v$ is the $d$th unit vector $v_d=(0,\ldots,0,1)$. Thus, for a  
given function $f$ on $\Om_+$ its even and odd extensions on $\Om$ with respect to the $d$th variable are 
$$
\EE f(\check x,x_d)=f(\check x,|x_d|)\qquad {\rm and} \qquad \OO f(\check x,x_d)={\rm sgn} (x_d)\, f(\check x,|x_d|),
$$
for $x=(\check x,x_d)\in\Om$, $x_d\neq0$. 
Also, for a function $F$ on $\Om$,  the even and odd parts of $F$ (with respect to the $d$th variable), are
$$F_{\rm even}(x)=(F(\check x,x_d)+F(\check x,-x_d))/2\qquad {\rm and} \qquad  F_{\rm odd}(x)=(F(\check x,x_d)-F(\check x,-x_d))/2.
$$
Recall, that we treat $F_{\rm even}$ and $F_{\rm odd}$ as the restrictions to $\Om_+$.

We begin with the first identity in \eqref{H1} and, proving the inclusion $\subset$  we follow the proof of \cite[Lemma 9.2]{Br} (see also \cite[Lemma 7.1.2]{D2}); we include details for the sake of completeness. 
Take $f\in H^1(\Om_+)$. We show that $\partial_j(\EE f)=\EE(\partial_i f)$, for $j=1,\ldots,d-1$, and $\partial_d(\EE f)=\OO(\partial_d f)$, which means that $\EE f\in H^1(\Om)$.
Fix $\v\in C^\infty_c(\Om)$ and let ${\rm supp}\,\v\subset B_R(0)$. Let $\eta\in C^\infty(0,\infty)$ be such that $\eta(t)=0$ for $0<t<1/2$ and $\eta(t)=1$ for $t>1$. Let $\eta_k(t)=\eta(kt)$.
For  $j=1,\ldots,d-1$ we write 
\begin{equation}\label{aux2}
\int_{\Om}\EE f\,\partial_j\v=2\int_{\Om_+}f\,\partial_j(\v_{\rm even}).
\end{equation}
Clearly, $\v_{\rm even}$ not necessarily is in $C^\infty_c(\Om_+)$, but $(\eta_k\v_{\rm even})(x):=\eta_k(x_d)\v_{\rm even}(x)$,  $x=(\check x,x_d)\in\Om_+$, is. Hence
$$
\int_{\Om_+}f\,\partial_j(\eta_k\v_{\rm even})=-\int_{\Om_+}\partial_jf\,\eta_k\v_{\rm even}.
$$
Noticing that $\partial_j(\eta_k\v_{\rm even})=\eta_k\,\partial_j(\v_{\rm even})$,  then letting $k\to \infty$ and using the Lebesgue dominated convergence theorem gives 
$$
\int_{\Om_+}f\,\partial_j(\v_{\rm even})=-\int_{\Om_+}\partial_jf\,\v_{\rm even}.
$$
Now we rewrite \eqref{aux2} to 
\begin{equation*}
\int_{\Om}\EE f\,\partial_j\v=-2\int_{\Om_+}\partial_j f\,\v_{\rm even}=-\int_{\Om}\EE(\partial_j f)\,\v.
\end{equation*}
This means that the weak $j$th derivative  of $\EE f$ in $L^2(\Om)$ is $\EE(\partial_j f)$. 

To treat the case $j=d$ we write
\begin{equation*}
\int_{\Om}\EE f\,\partial_d\v=2\int_{\Om_+}f\,\partial_d(\v_{\rm odd}).
\end{equation*}
Since $\v_{\rm odd}(\check{x},0)=0$ for $(\check{x},0)\in\Om$, and ${\rm supp}\,\v\subset B_R(0)$ for some $R>0$, hence there exists $M>0$ such that for $(\check{x},x_d)\in\Om$ we have
$$
|\v_{\rm odd}(\check{x},x_d)|\le M|x_d|.
$$
Clearly, $\v_{\rm odd}$  not necessarily is in $C^\infty_c(\Om_+)$, however $\eta_k\v_{\rm odd}$ is. Therefore,
$$
\int_{\Om_+}f\,\partial_d(\eta_k\v_{\rm odd})=-\int_{\Om_+}\partial_df\,\eta_k\v_{\rm odd}.
$$
But $\partial_d(\eta_k\v_{\rm odd})= \eta_k\partial_d \v_{\rm odd}+ k\eta'(kx_d)\v_{\rm odd}$ and we claim that 
 \begin{equation}\label{aux3}
k\int_{\Om_+}f\, \eta'(kx_d)\v_{\rm odd}\to 0
\end{equation}
with $k\to\infty$. Indeed, if $C=\|\eta'\|_\infty$, then 
$$
k\Big|\int_{\Om_+}f\, \eta'(kx_d)\v_{\rm odd}\Big|\le kCM\int_{A(k,R)}|f|x_d\,dx\le CM\int_{A(k,R)}|f|\,dx,
$$
where $A(k,R)= B_R(0)\cap\{(\check x,x_d)\in \Om_+\colon 0<x_d<1\slash k\}$,
and the last quantity tends to 0 as $k\to\infty$. This means that letting $k\to\infty$ shows that
$$
\int_{\Om_+}f\,\partial_d(\v_{\rm odd})=-\int_{\Om_+}\partial_df\,\v_{\rm odd}
$$
and hence 
\begin{equation*}
\int_{\Om}\EE f\,\partial_d\v=-2\int_{\Om_+}\partial_i f\,\v_{\rm odd}=-\int_{\Om}\OO(\partial_d f)\,\v.
\end{equation*}
This proves that the weak $d$th derivative  of $\EE f$ in $L^2(\Om)$ is $\OO(\partial_d f)$ and finishes the proof of
the inclusion $\subset$ in  \eqref{H1}. 

To prove the opposite inclusion for the first identity in  \eqref{H1}, take $F\in H^1(\Om)$. Without any loss of generality we can 
assume that $F$ is even (otherwise, take $F_{\rm even}$ treated at this moment as a function on $\Omega$; clearly, $F_{\rm even}\in H^1(\Om)$ and even parts of $F$ and $F_{\rm even}$ 
coincide). Since $C^\infty_c(\Om_+)\subset C^\infty_c(\Om)$, hence $F_{\rm even}=F|_{\Om_+}\in H^1(\Om_+)$. 

We now pass to the second identity in \eqref{H1} and prove the inclusion $\subset$. Fix $f\in H^1_0(\Om_+)$ and take a sequence $\v_n\in C^\infty_c(\Om_+)$ 
such that $\v_n\to f$ in $H^1(\Om_+)$. This means, in particular, that $\{\v_n\}$ is a Cauchy sequence in $H^1(\Om_+)$, and hence $\{\OO\v_n\}$ is a Cauchy sequence in $H^1(\Om)$.
Let $F$ be the limit of $\OO\v_n$ in $H^1(\Om)$. Since $\OO\v_n\in C^\infty_c(\Om)$, we have that $F\in H^1_0(\Om)$. It is also clear that $F$ is odd on $\Om$ and $F|_{\Om_+}=f$. This shows that $f\in (H^1_0(\Om))_{\rm odd}$.

To prove the opposite inclusion for the second identity in  \eqref{H1} first note that if $F\in H^1(\Om)$, then $F|_{\Om_+}\in H^1(\Om_+)$. This is because $C^\infty_c(\Om_+)\subset C^\infty_c(\Om)$, 
and hence  $\partial_j(F|_{\Om_+})=(\partial_j f)|_{\Om_+}$, $j=1,\ldots,d$. Thus, if we fix  $F\in H^1_0(\Om)$, then $F_{\rm odd}\in H^1(\Om_+)$.
It remains to verify that $F_{\rm odd}\in H^1_0(\Om_+)$. Take a sequence $\Phi_n\in  C^\infty_c(\Om)$ such that $\Phi_n\to F$ in $H^1(\Om)$. Then also $(\Phi_n)_{\rm odd}\to F_{\rm odd}$ in $H^1(\Om)$. 
Now it suffices to check that if $\Phi\in  C^\infty_c(\Om)$ is odd, then $\v:=\Phi|_{\Om_+}$ can be approximated by functions from $C^\infty_c(\Om_+)$ in $H^1(\Om_+)$. Fix such $\Phi$, 
take the same $\eta$ as before, consider $\eta_k\,\v\in C^\infty_c(\Om_+)$ 
and write
$$
\|\eta_k\v-\v\|_{H^1(\Om_+)}^2=\|\v(\eta_k-1)\|^2_{L^2(\Om_+)}+\sum_{j=1}^d\|\partial_j(\v(\eta_k-1))\|^2_{L^2(\Om_+)}.
$$
It is clear that $\|\v(\eta_k-1)\|^2_{L^2(\Om_+)}\to0$ as $k\to \infty$. For the remaining terms note that for $j=1,\ldots,d-1$ we have $\partial_j(\v(\eta_k-1))=(\partial_j\v)(\eta_k-1))$ and hence 
again $\|\partial_j(\v(\eta_k-1))\|^2_{L^2(\Om_+)}\to0$ as $k\to \infty$. For $j=d$, 
$$\partial_d(\v(\eta_k-1))=(\partial_d\v)(\eta_k-1))+k\eta'(kx_d)\v.$$
Therefore it remains to check that
$$
k^2\int_{\Om_+}|\eta'(kx_d)\v(x)|^2\,dx \to 0
$$
as $k\to \infty$. But this is done by an argument analogous to that used for \eqref{aux3}. This finishes  the proof 
the second identity in \eqref{H1} and completes the proof of Lemma \ref{lem:first}. 
\end{proof}
%%%%%%%%%%%%%%%%%%%%%%%%%%%%%%%%%%%%%%%%%%%
It is worth noticing here that for $f\in H^1(\Omega_+)$, if $F\in H^1(\Omega)$ is such that $f= F_{\rm even}$, then we can simply assume that $F$ is even (see a comment in the proof of
Lemma \ref{lem:first}). Analogous remark applies to $f\in H^1_0(\Omega_+)$.

Recall that in the setting of a sesquilinear form $\mathfrak{t}$ with domain $\D(\mathfrak{t})$, defined on a Hilbert space $(\mathcal H, \langle\cdot,\cdot\rangle)$, the associated operator
$A_\mathfrak{t}$ is defined by $A_\mathfrak{t}h=u_h$, where $h\in \D(A_\mathfrak{t})$ and
$$
\D(A_\mathfrak{t})=\{h\in \D(\mathfrak{t})\colon \exists u_h\in \mathcal H\,\,\forall h'\in \D(\mathfrak{t})\,\, \mathfrak{t}[h,h']=\langle u_h,h'\rangle\}.
$$

%%%%%%%%%%%%%%%%%%%%%%%%%%%%%%%%%%%%%%%%%%%
\begin{propo} \label{pro:one}
We have
\begin{equation}\label{dom1}
\D(-\Delta_{N,\Om_+})=\big(\D(-\Delta_{N,\Om})\big)_{\rm even}
\end{equation}
and
\begin{equation}\label{for1}
(-\Delta_{N,\Om_+})(F_{\rm even})=\big((-\Delta_{N,\Om})F\big)_{\rm even}, \quad {\rm for}\quad F\in \D(-\Delta_{N,\Om}).
\end{equation}
Similarly,
\begin{equation}\label{dom2}
\D(-\Delta_{D,\Om_+})=\big(\D(-\Delta_{D,\Om})\big)_{\rm odd}
\end{equation}
and
\begin{equation}\label{for2}
(-\Delta_{D,\Om_+})(F_{\rm odd})=\big((-\Delta_{D,\Om})F\big)_{\rm odd}, \quad {\rm for}\quad F\in \D(-\Delta_{D,\Om}).
\end{equation}
\end{propo}
%%%%%%%%%%%%%%%%%%%%%%%%%%%%%%%%%%%%%%%%%%%
\begin{proof}
We consider only the case of the Neumann Laplacians and prove \eqref{dom1} and \eqref{for1}; the arguments leading to \eqref{dom2} and \eqref{for2} are analogous. 
For simplicity of notation till the end of this proof we write $-\Delta$ and $-\Delta_+$ instead of $-\Delta_{N,\Om}$ and $-\Delta_{N,\Om_+}$, correspondingly. 
Analogously, we write $\mathfrak{t}$ and $\mathfrak{t}_+$ rather than $\mathfrak{t}_{N,\Om}$ and $\mathfrak{t}_{N,\Om_+}$. Recall, that for simplicity we also assume that 
$v_d=(0,\ldots,0,1)$.

We first prove the inclusion $\subset$ in \eqref{dom1}. Take $f\in\D(-\Delta_+)$. Hence $f\in H^1(\Om_+)$ and there 
is $u_f\in L^2(\Om_+)$ such that for any $g\in H^1(\Om_+)$ we have
\begin{equation}\label{eqq1}
\mathfrak{t}_+[f,g]=\langle u_f,g\rangle_{L^2(\Om_+)},
\end{equation}
which also means that $(-\Delta_+)f=u_f$. Consider $\EE f$ which, by Lemma \ref{lem:first}, is in $H^1(\Om)$.
We shall verify that for every $G\in H^1(\Om)$ it holds
\begin{equation}\label{eqq2}
\mathfrak{t}[\EE f,G]=\langle \EE(u_f),G\rangle_{L^2(\Om)},
\end{equation}
which will mean that  $\EE f\in \D(-\Delta)$ and hence  $f\in \big(\D(-\Delta)\big)_{\rm even}$, 
and also that $(-\Delta)(\EE f)=\EE(u_f)$, which implies that $((-\Delta)F)_{\rm even}=(-\Delta_+)(F_{\rm even})$ for $F\in\D(-\Delta)$.

For any $G\in H^1(\Om)$ we have 
\begin{align*}
\mathfrak{t}[\EE f,G]
&=\sum_{j=1}^{d-1}\int_{\Om}\partial_j(\EE f)\,\overline{\partial_j G}+
\int_{\Om}\partial_d (\EE f)\,\overline{\partial_d G}\\
&=\sum_{j=1}^{d-1}\int_{\Om}\EE (\partial_j f)\,\overline{\partial_j G}+
\int_{\Om}\OO( \partial_d f)\,\overline{\partial_d G}\\
&=2\sum_{j=1}^{d-1}\int_{\Om_+}\partial_j f\,\overline{(\partial_j G)_{\rm even}}+
2\int_{\Om_+}\partial_d f\,\overline{(\partial_d G)_{\rm odd}}\\
&=2\sum_{j=1}^{d-1}\int_{\Om_+}\partial_j f\,\overline{\partial_j(G_{\rm even})}+
2\int_{\Om_+}\partial_d f\,\overline{\partial_d(G_{\rm even})}\\
&=2\mathfrak{t}_+[f,G_{\rm even}].
\end{align*}
On the other hand,
$$
\langle \EE(u_f),G\rangle_{L^2(\Om)}=2\langle u_f,G_{\rm even}\rangle_{L^2(\Om_+)},
$$
and hence inserting $G_{\rm even}$ for $g$ in \eqref{eqq1} gives \eqref{eqq2}; note that $G_{\rm even}\in H^1(\Om_+)$ by Lemma \ref{lem:first}.

To prove the opposite inclusion, take $F\in\D(-\Delta)$. Hence $F\in H^1(\Om)$ and there is $U_F\in L^2(\Om)$ such that for any $G\in H^1(\Om)$ we have
\begin{equation}\label{eq1}
\mathfrak{t}[F,G]=\langle U_F,G\rangle_{L^2(\Om)},
\end{equation}
which also means that $(-\Delta)F=U_F$. We shall verify that for every $g\in H^1(\Om_+)$ it holds
\begin{equation}\label{eq2}
\mathfrak{t}_+[F_{\rm even},g]=\langle (U_F)_{\rm even},g\rangle_{L^2(\Om_+)},
\end{equation}
which will mean that $F_{\rm even}\in \D(-\Delta_+)$ (note that $F_{\rm even}\in H^1(\Omega_+)$ by Lemma \ref{lem:first}) and that $(-\Delta_+)(F_{\rm even})
=((-\Delta)F)_{\rm even}$ for $F\in\D(-\Delta)$.

For any $g\in H^1(\Om_+)$ we have 
\begin{align*}
2\mathfrak{t}_+[F_{\rm even},g]&=2\sum_{j=1}^{d-1}\int_{\Om_+}\partial_j(F_{\rm even})\,\overline{\partial_j g}+2\int_{\Om_+}\partial_d(F_{\rm even})\,\overline{\partial_d g}\\
&=2\sum_{j=1}^{d-1}\int_{\Om_+}(\partial_j F)_{\rm even}\,\overline{\partial_j g}+2\int_{\Om_+}(\partial_d F)_{\rm odd}\,\overline{\partial_d g}\\
&=\sum_{j=1}^{d-1}\int_{\Om}\partial_j F\,\overline{\EE(\partial_j g)}+\int_{\Om}\partial_d F\,\overline{\OO(\partial_d g)}\\
&=\sum_{j=1}^{d-1}\int_{\Om}\partial_j F\,\overline{\partial_j(\EE g)}+\int_{\Om}\partial_d F\,\overline{\partial_d(\EE g)}\\
&=\mathfrak{t}[F,\EE g].
\end{align*}
On the other hand,
$$
2\langle (U_F)_{\rm even},g\rangle_{L^2(\Om_+)}=\langle U_F,\EE g\rangle_{L^2(\Om)},
$$
and hence inserting $\EE g$ for $G$ in \eqref{eq1} gives \eqref{eq2}; note that $\EE g\in H^1(\Om)$ by Lemma \ref{lem:first}. This completes the proof of \eqref{eq2} and thus the conclusion following it and hence finishes the proof of \eqref{dom1} and \eqref{for1}.
\end{proof}

Before proceeding to the proof of Theorem \ref{thm:main} we need to make some preparatory comments. It is well known that, as a part of the spectral theorem,
the following commuting property of the functional calculus holds: if $A$ is a self-adjoint operator on a Hilbert space $\mathcal H$ and $B$ is a bounded operator on 
$\mathcal H$ such that $BA\subset AB$, then also $B\Psi(A)\subset \Psi(A)B$, for any Borel function $\Psi$ on $\mathbb R$. In addition, if $\Psi$ is bounded, then $\Psi(A)$
is a bounded operator and the latter inclusion becomes the identity. See, for instance, \cite[Theorem 4.1 (d), p.\,323]{Co} specified to self-adjoint operators, or \cite{RS}.

We shall need the following two-Hilbert space and two-operator version of the above. Namely, if $A_1$ and $A_2$ are self-adjoint operators on  Hilbert spaces $\mathcal H_1$ and
$\mathcal H_2$, respectively, and $B\colon \mathcal H_1\to \mathcal H_2$ is a bounded operator such that $BA_1\subset A_2B$, then also
$B\Psi(A_1)\subset \Psi(A_2)B$, for any Borel function $\Psi$ on $\mathbb R$.
Again, if $\Psi$ is bounded, then the last inclusion becomes the identity. Such a version is known, at least as a folklore, but it is hard to find it in the literature in the above formulation.
However, see \cite[Proposition 5.15]{Sch}, where it is said that in the above mentioned setting the condition  $BA_1\subset A_2B$ is equivalent with the condition $BE_1(M)=E_2(M)B$, where $M$
is an arbitrary Borel subset of $\mathbb{R}$, and $E_i$ denotes the spectral measure corresponding to $A_i$, $i=1,2$. This equivalent condition easily implies the claim, that is
$B\Psi(A_1)\subset \Psi(A_2)B$, for any Borel function $\Psi$ on $\mathbb R$.

We take an opportunity to point out that the version we need can be also inherited from the usual property
of the functional calculus of one self-adjoint operator by considering the direct sum $A_1\oplus A_2$ on $\mathcal H_1\oplus \mathcal H_2$ and taking as a bounded operator
on $\mathcal H_1\oplus \mathcal H_2$ the operator $(x,y)\mapsto (0,Bx)$. (Checking that we indeed end up in the one-Hilbert space setting with all necessary assumptions satisfied, and the
conclusion from the one-Hilbert space version implies the desired inclusion is straightforward.)
This argument, that changes the intertwining condition onto the
commuting condition, is known as Berberian's trick; we owe this information to Professor Jan Stochel to whom we are very indebted.

\textbf{Proof of Theorem \ref{thm:main}.}
We consider only the case of the Neumann Laplacians and prove \eqref{kerN}; the arguments leading to \eqref{kerD} are analogous.
As in the proof of Proposition \ref{pro:one}, for simplicity of notation till the end of this proof we write $-\Delta$ and $-\Delta_+$ instead of $-\Delta_{N,\Om}$ and $-\Delta_{N,\Om_+}$,  and consequently, $\Phi(-\Delta)$ and $\Phi(-\Delta_+)$ instead of $\Phi(-\Delta_{N,\Om})$ and $\Phi(-\Delta_{N,\Om_+})$, correspondingly. Analogously, we write $K^\Phi$  rather than 
$K^\Phi_{-\Delta_{N,\Om}}$. Keeping in mind delicacies usually connected to domains of unbounded operators we decided to be slightly pedantic in what follows.

The reflection $\s_v$ induces a natural action on functions defined on $\Om$: if $F$ is such a function, then $\check F(x):= F(\s_v(x))$, $x\in\Om$.
As an easy calculation shows, the mapping $\check \,\colon L^2(\Om)\to L^2(\Om)$ leaves $\D(-\Delta)$ invariant, and hence it is a bijection on $\D(-\Delta)$.
This implies that 
\begin{equation}\label{com}
(-\Delta F)\check{\,}=-\Delta\check{F}, \qquad F\in\D(-\Delta).
\end{equation}
Thus, by the spectral theorem, also 
\begin{equation*}%\label{com}
(\Phi(-\Delta) F)\check{\,}=\Phi(-\Delta)\check{F}, \qquad F\in\D(\Phi(-\Delta)),
\end{equation*}
and, consequently, since $\D(\Phi(-\Delta))$ is dense in $L^2(\Om)$, for the kernel $K^\Phi$ we have
\begin{equation}\label{symm}
K^\Phi(\s_v(x),y)=K^\Phi(x,\s_v(y)),\qquad (x,y)\in\Om\times\Om\,-\, a.e.
\end{equation}
On the other hand, by using Proposition \ref{pro:one}, it is also clear that
\begin{equation*}%\label{com}
(-\Delta_{+})(F_{\rm even})=(-\Delta F)_{\rm even}, \qquad F\in\D(-\Delta)
\end{equation*}
and hence, the comment made above applied to $\mathcal H_1=L^2(\Om)$ and  $\mathcal H_2=L^2(\Om_+)$, $A_1=-\Delta$ and $A_2=-\Delta_{+}$, and $B\colon L^2(\Om)\to L^2(\Om_+)$ defined
by $F\mapsto F_{\rm even}$, gives for every $\Phi$
\begin{equation*}%\label{com}
\Phi(-\Delta_{+})(F_{\rm even})=(\Phi(-\Delta) F)_{\rm even}, \qquad F\in\D(\Phi(-\Delta))\/.
\end{equation*}

Thus, given $f\in \D(\Phi(-\Delta_+))$, take $F\in \D(\Phi(-\Delta))$ such that $F_{\rm even}=f$; we can assume that $F$ is even. Then for $x\in\Om_+$ we obtain
\begin{align*}
\Phi(-\Delta_+)f(x)
&=\frac12\big(\Phi(-\Delta)F(x)+\Phi(-\Delta)F(\s_v(x))\big)\\
&=\frac12\Big(\int_\Om K^\Phi(x,y)F(y)\,dy+\int_\Om K^\Phi(\s_v(x),y)F(y)\,dy\Big)\\
&=\frac12\Big(\int_{\Om_+} \Big[K^\Phi(x,y)+ K^\Phi(x,\s_v(y))\Big]f(y)\,dy\\
&\,\,\,\,\,\,\,\,\,\,+\int_{\Om_+} \Big[K^\Phi(\s_v(x),y)+ K^\Phi(\s_v(x),\s_v(y))\Big]f(y)\,dy\Big)\\
&=\int_{\Om_+} \Big[K^\Phi(x,y)+ K^\Phi(x,\s_v(y))\Big]f(y)\,dy,
\end{align*}
where, for the last identity, we used \eqref{symm} (and  $K^\Phi(\s_v(x),\s_v(y))=K^\Phi(x,y)$ that follows from \eqref{symm}).
This means that $\Phi(-\Delta_+)$ has an integral kernel and \eqref{kerN} takes place. The proof of Theorem \ref{thm:main} is completed.

%%%%%%%%%%%%%%%%%%%%%%%%%%%%%%%%%%%%%%%%%%%%%%%%%%%%%%%%%%%%%%%%%%%%%%%%%%%%
\section{Probabilistic approach} \label{sec:prob}
%%%%%%%%%%%%%%%%%%%%%%%%%%%%%%%%%%%%%%%%%%%%%%%%%%%%%%%%%%%%%%%%%%%%%%%%%%%%
The reflection principle appears in the theory of stochastic processes and refers to properties of a Wiener process (Brownian motion). Both theories are linked by the fact that the Laplace operator is the infinitesimal generator of a transition semigroup of the Wiener process and the operators $-\Delta_{D,\Omega}$ and $-\Delta_{N,\Omega}$ refer to a Wiener process killed upon leaving $\Omega$ and a reflected Wiener process, respectively. In this section we present the refection principle from the point of view of a killed Wiener process and strong Markov property. 

Let $W=(W_t)_{t\geq 0}$ be a $d$-dimensional Wiener process starting from $x\in\R$ and denote by $\pr^x$, $\ex^x$ and $\mathcal{F}=(\mathcal{F}_t)_{t\geq 0}$ the corresponding probability distribution, expecting value and the filtration generated by $W$. We will simply write $\pr$ and $\ex$ whenever $x=0$. Recall that $\pr^x$ is absolutely continuous with respect to Lebesgue measure and 
$$
\pr^x(W_t\in dy)/dy = \frac{1}{(2\pi t)^{d/2}}\exp\left(-\frac{\|x-y\|^2}{2t}\right)\/,
$$
which is just $p_{t/2}(x,y)$. To distinguish the probabilistic approach from the previous one we will write $\g_t(x,y):=p_{t/2}(x,y)$. 

For a given nonempty open set $\Omega\subset \R$ we define the first exit time of $W$ from $\set$ by
\begin{equation*}
  \tau_\set = \inf\{t: W_t\notin \set\}\/.
\end{equation*} 
Continuity of paths implies that $\tau_\set$ is an $(\mathcal{F}_t)$-Markov stopping time. We denote by $W^\set=(W^\set_t)_{t\geq 0}$ the process killed upon leaving the set $\set$ and write $\gD_t(x,y)$ for its transition density function, i.e. 
$$
\gD_t(x,y) = \ex^x[t<\tau_\set,W_t\in dy]/dy\/.
$$
By the strong Markov property, we can describe it in the following way 
$$
\gD_t(x,y) = \g_t(x,y)-\ex^x[t>\tau_\set; \g(t-\tau_\set,W_{\tau_\set},y)]\/,\quad x,y\in\set\/.
$$
The identity given above is often called the Hunt formula. 

The classical reflection principle in $\mathbb{R}$ is a consequence of a strong Markov property and it states that for a given stopping time $\tau$ the process
\begin{eqnarray*}
   W^\tau_t = \left\{
	\begin{array}{rl}
	   W_t\/,& t\leq \tau\\
		 2W_\tau-W_t\/, & t>\tau
	\end{array}
	\right.
\end{eqnarray*}
is also a Wiener process. Note that the paths of $W^\tau$ are glued with the original trajectory of $W_t$ (up to time $\tau$) and the trajectory reflected with respect to a line $y=W_\tau$ (after $\tau$). Applying the result to the special case $\tau=\inf\{t: W_t=a\}$, $a>0$, we obtain 
\begin{eqnarray*}
   \pr(\sup_{0\leq s\leq t}W_s\geq a) = 2\pr(W_t>a)\/.
\end{eqnarray*}
This essentially weaker relation is also often called the reflection principle.

We will study another consequence of the strong Markov property. We establish the relation between the transition density functions of an open set $\Om\subset\R$, 
which is symmetric in $\langle v\rangle^\bot$, and its positive part $\set_+$.
Let us also denote $\set_b = \set\setminus(\set_+\cup \set_-)$, where $\set_{-}=\s_v(\set_+)$.
%%%%%%%%%%%%%%%%%%%%%%%%%%%%%%%%%%
\begin{propo}
  Let $\set$ and $\set_+$ be as described above. Then
   \begin{eqnarray*}
	    \g^{\set_+}_t(x,y) = \g^{\set}_t(x,y)-\g^{\set}_t(\sigma_v(x),y)\/,\quad x,y\in\set_+\/.
	\end{eqnarray*}
\end{propo}
%%%%%%%%%%%%%%%%%%%%%%%%%%%%%%%%%%%
\begin{proof}
Since $\set_+\subset\set$ we obviously have $\tau_{\set_+}\leq \tau_{\set}$ and for a given Borel set $A\subset \set_+$ we have
$$
\pr^x[t<\tau_{\set_+},W_t\in A] = \pr^x[t<\tau_\set,W_t\in A]-\pr^x[\tau_{\set_+}\leq t <\tau_\set, W_t\in A].
$$
	Note that $\tau_{\set_+}< \tau_{\set}$ if and only if $W_{\tau_{\set_+}}\in\set_b$ and consequently, using the strong Markov property, we get
$$
	   \pr^x[\tau_{\set_+}\leq t <\tau_\set, W_t\in A] = \pr^x[W_{\tau_{\set_+}}\in\set_b\/, \tau_{\set_+}\leq t\/, \ex^{W_{\tau_{\set_+}}}[t-\tau_{\set_+}<\tau_\set,W_{t-\tau_{\set_+}}\in A]].
$$
Note that $(W_{\tau_{\set_+}+s})_{s\geq 0}$ is a Wiener process starting from a point $W_{\tau_{\set_+}}\in\set_b$ and $(W_{\tau_{\set_+}+s})_{s\geq 0}-W_{\tau_{\set_+}}$ is independent from 
$\tau_{\set_+}$. Consequently $\sigma_v(W_{\tau_{\set_+}+s})$ is also a Wiener process starting from the same point. Moreover, the first exit time from $\set$ for both processes are the same 
due to the symmetry of $\set$. Since $\{\sigma_v(W_{t-\tau_{\set_+}})\in A\} = \{W_{t-\tau_{\set_+}}\in \sigma_v(A)\}$ we can simply rewrite the last above-given expression as
$$
\pr^x[W_{\tau_{\set_+}}\in\set_b\/, \tau_{\set_+}\leq t\/, \ex^{W_{\tau_{\set_+}}}[t-\tau_{\set_+}<\tau_\set,W_{t-\tau_{\set_+}}\in \sigma_v(A)]].
$$
	Thus, the strong Markov property implies that
$$
\pr^x[\tau_{\set_+}\leq t <\tau_\set, W_t\in A] =\pr^x[\tau_{\set_+}\leq t <\tau_\set, W_t\in \sigma_v(A)].
$$
	Note that we can drop the condition $\tau_{\set_+}\leq t$ since $\sigma_v(A)\subset \set_{-}$ and consequently $W_t\notin \set_+$ implies $\tau_{\set_+}<t$. Once again we can consider $\sigma_v(W)$ instead of $W$ and using the symmetry of $\set$ arrive at
  \begin{align*}
	  \pr^x[t<\tau_{\set_+},W_t\in A] &= \pr^x[t<\tau_\set,W_t\in A]-\pr^x[ t <\tau_\set, W_t\in \sigma_v(A)]\\
		&= \pr^x[t<\tau_\set,W_t\in A]-\pr^{\sigma_v(x)}[ t <\tau_\set, W_t\in A]\/,
	\end{align*}
	which ends the proof.
\end{proof}

%%%%%%%%%%%%%%%%%%%%%%%%%%%%%%%%%%%%%%%%%%%%%%%%%%%%%%%%%%%%%%%%%%%%%%%%%%
\section{Applications and Examples} \label{sec:ex}
%%%%%%%%%%%%%%%%%%%%%%%%%%%%%%%%%%%%%%%%%%%%%
In this section we first extend Theorem \ref{thm:main} to the setting, where a single reflection in $\R$ is replaced by $k$ reflections $1\le k\le d$, associated with mutually 
orthogonal $k$ vectors. Then we discuss applications of Theorem \ref{thm:main}
including resolvents, Riesz potentials, Green's functions and heat semigroups associated to the Neumann/Dirichlet Laplacians on open sets of $\R$. These result in reflection principle formulas for the corresponding integral kernels, i.e. resolvent kernels, Riesz potential kernels, Green's functions, and heat kernels. We also discuss concrete examples for several symmetric open sets $\Om$ by comparing formulas for the Neumann/Dirichlet heat kernels for $-\Delta_{N/D,\Om}$ and $-\Delta_{N/D,\Om_+}$ which are known to be given in terms of series.

%%%%%%%%%%%%%%%%%%%%%%%%%%%%%%%%%%%%%%%%%%%%%
\subsection{Reflection principles for orthogonal root systems}
Theorem \ref{thm:main} easily leads to a corollary, where symmetries related to a reflection group associated with an orthogonal root system are involved.
Recall that a \textit{(normalized) root system} in $\R$ is a finite set $R$ of unit vectors such that $\s_\a(R)=R$ for every $\a\in R$. 
Clearly, $R\cap\mathbb{R}\a=\{\a,-\a\}$ for every $\a\in R$. The subgroup of $O(d)$ generated by the reflections $\s_\a$, $W=W(R)={\rm gp}\,\{\s_\a\colon \a\in R\}$, is called the \textit{finite reflection group}, or Weyl group, \textit{associated with} $R$. 
A choice of $\a\in \R$ such that $\langle \a,\check\a\rangle\neq0$ for every $\a\in R$, gives the partition $R=R_+\sqcup R_-$, where $R_+=\{\a\in R\colon \langle \a,\check\a\rangle>0\}$ and $R_-=\s_{\check\a}(R_+)$; 
$R_+$ is then referred to as the \textit{set of positive roots}. This partition distinguishes $C_+=\{x\in\R\colon \forall\, \a\in R_+ \,\,\,\langle x,\a\rangle>0\}$, which is called 
the \textit{positive Weyl chamber}. A root system $R$ is called orthogonal if $R_+$ is orthogonal as a set of vectors (this does not depend on the choice of $\check\a$). 
For a comprehensive treatment of the general theory of finite reflection groups the reader is kindly referred to \cite{Hu}.

Given an orthogonal root system $R$ in $\R$ with $R_+$ as a set of positive roots, without any loss of generality (possibly by rotating and permutating the coordinate axes) we can
assume that $R_+=\{e_1,\ldots,e_k\}$, where $k\in\{1,\ldots,d\}$, and $e_j$ is the $j$th coordinate unit vector. Thus, given $1\le k\le d$ let $R_+^{(k)}=\{e_j\colon j=1,\ldots,k\}$ be
the system of positive roots so that $R^{(k)}=R_+^{(k)}\sqcup(-R_+^{(k)})$ is the orthogonal root system in $\R$.
The corresponding positive Weyl chamber is $\R_{k,+}:=\{x\in\R\colon x_i>0\,\, {\rm for}\,\, i=1,\ldots,k\}$. Together with $\R_{k,+}$ consider the open sets
$$
\R_{j,+}:=\{x\in\R\colon x_i>0\,\, {\rm for}\,\, i=1,\ldots,j\},\qquad j=k-1,\ldots,0,
$$
so that $\R_{0,+}=\R$ and $\R_{j,+}$ is a 'half' of $\R_{j-1,+}$ in the sense that
$$
\R_{j-1,+}=\R_{j,+}\cup \sigma_{e_{j-1}}(\R_{j,+}) \cup \{x\in\R\colon x_i>0\,\, {\rm for}\,\, i=1,\ldots,j-1 \,\,{\rm and} \,\,x_j=0\}.
$$ 
Let $\Omega\subset \R$ be an open set, symmetric with respect to the Weyl group associated with $R^{(k)}$. This is equivalent with the statement that $\sigma_{e_{j}}(\Omega)=\Omega$ for
$j=1,\ldots,k$. Let $\Omega_j=\Omega\cap \R_{j,+}$. Applying succesively Theorem \ref{thm:main} to the sets $\Om=\Om_0,\Om_1,\ldots,\Om_k$, allows to express the kernels associated with 
$\Phi(-\Delta_{N/D,\Om_k})$ through the kernels associated with $\Phi(-\Delta_{N/D,\Om})$. Here we use the notation: for $x=(x_1,\ldots,x_d)\in \R_{k,+}$ and $\varepsilon\in\{-1,1\}^k$ 
we write $\varepsilon x=(\varepsilon_1 x_1,\ldots \varepsilon_k x_k,x_{k+1},\ldots,x_d)$ and ${\rm sgn}(\varepsilon)=\prod_{i=1}^k\varepsilon_i$.
%%%%%%%%%%%%%%%%%%%%%%%%%%%%%%%%%%%%%%%%%%%
\begin{corollary} \label{cor:three}
Let $\Om$ be an open subset of $\R$ symmetric with respect to $\langle e_j\rangle^\bot$, $j=1,\ldots,k$, with $\Om_{k,+}=\Om\cap \R_{k,+}$ as its positive chamber.
Let $\Phi$ be a Borel function on $[0,\infty)$. Assume that $\Phi(-\Delta_{N/D,\Om})$ is an integral operator with the kernel $K^\Phi_{-\Delta_{N/D,\Om}}$. Then $\Phi(-\Delta_{N/D,\Om_{k,+}})$ 
is also an integral operator with the kernel $K^\Phi_{-\Delta_{N/D,\Om_{k,+}}}$ given by
\begin{equation*}%\label{kerND}
K^\Phi_{-\Delta_{N,\Om_{k,+}}}(x,y)=\sum_{\varepsilon\in\{-1,1\}^k}K^\Phi_{-\Delta_{N,\Om}}(\varepsilon x,y),\qquad x,y\in\Om_{k,+},
\end{equation*} 
or by 
\begin{equation*}%\label{kerND}
K^\Phi_{-\Delta_{D,\Om_{k,+}}}(x,y)=\sum_{\varepsilon\in\{-1,1\}^k}{\rm sgn}(\varepsilon)K^\Phi_{-\Delta_{D,\Om}}(\varepsilon x,y),\qquad x,y\in\Om_{k,+},
\end{equation*} 
respectively.
\end{corollary}
%%%%%%%%%%%%%%%%%%%%%%%%%%%%%%%%%%%%%%%%%%%
Notice that the above corollary generalizes the result of Theorem \ref{thm:main} (the case of $k=1$, up to a rotation of coordinate axes). Notice also that the formulas for $p_t^{N/D,\pi/2}(x,y)$ discussed after the statement of Corollary \ref{cor:one} 
are consistent with the formulas given in Corollary \ref{cor:three} (the case of $k=d=2$). 

%%%%%%%%%%%%%%%%%%%%%%%%%%%%%%%%%%%%%%%%%%%%%
\subsection{Resolvent kernels, Riesz potentials and Green's functions}
Although the considerations that follow could be carried on in the setting of a general open set $\Om$, we concentrate the attention on the case $\Om=\R$. This allows us to write 
several relevant formulas in their closed forms.

Recall that we have  $-\Delta_{N,\R}=-\Delta_{D,\R}$, and hence the kernels of the operators $\Phi(-\Delta_{N/D,\R})$,
if exist, are identical. Therefore, in what follows, in the case of $\Om=\R$ we skip the characters $N$ and $D$,
and the symbol $\R$,
and denote the  resolvent kernels, Riesz potential kernels, and Green's functions 
corresponding to $-\Delta$ simply by $r_\lambda$,  $R_\sigma$, and $G$,
respectively. When it comes to the analogous kernels  associated with the half-space $\R_+$, $d\ge1$, we keep the convention used in  Section \ref{sec:intro} (related to the heat kernels) and simply write $r_\lambda^{N/D,+}$,  $R^{N/D,+}_\sigma$, and $G^{N/D,+}$.

Considering in Theorem \ref{thm:main} $\Phi_\la(u)=(u+\lambda)^{-1}$, $u\geq 0$, $\lambda> 0$, we arrive at the corresponding resolvent operator $\mathcal R_\la=(-\Delta+\la I)^{-1}$. Then the resolvent kernel $r_\la$, if exists, is given by
$$
r_\lambda(x,y)= \int_{0}^\infty e^{-\lambda t}p_t(x,y)\,dt.
$$
It is easily seen that for any $d\ge1$ and $x,y\in\R$, $x\neq y$, the above integral converges and we have  (see \cite[8.432 (7)]{GR})
$$
r_\lambda(x,y) =  \frac{\lambda^{(d-2)/4}}{2\pi^{d/2}}\,\frac{K_{d/2-1}(\|x-y\|\sqrt{\lambda})}{\|x-y\|^{d/2-1}}\/,
$$
where $K_\nu(z)$ is the modified Bessel function of the second kind of order $\nu$ (also called Macdonald's function). Therefore, from Theorem \ref{thm:main} we directly obtain for $x,y\in \mathbb{R}_+^d$, $x\neq y$,
\begin{equation}\label{d1}
r_\lambda^{N/D,+}(x,y) =\frac{\lambda^{(d-2)/4}}{2\pi^{d/2}}\left(\frac{K_{d/2-1}(\|x-y\|\sqrt{\lambda})}{\|x-y\|^{d/2-1}}\pm\frac{K_{d/2-1}(\|\tilde{x}-y\|\sqrt{\lambda})}{\|\tilde{x}-y\|^{d/2-1}}\right)\/,
\end{equation}
where
$\tilde{x}=(x_1,\ldots,x_{d-1},-x_d)$ for $x=(x_1,\ldots,x_d)\in\R_+$, and the sign from the symbol $\pm$ is chosen accordingly to the choice of $N$ or $D$. Since $K_{-1/2}(u)=\sqrt{\pi/(2u)}\, e^{-u}$, therefore, in dimension 1, $r_\lambda(x,y) = e^{-\sqrt{\lambda}|x-y|}/(2\sqrt{2\lambda})$, and \eqref{d1} specified to $d=1$ takes the form
$$
r_\lambda^{N/D,+}(x,y) = \frac 1{\sqrt{2\la}}e^{-\sqrt{2\la}x}[\cosh/\sinh] (\sqrt{2\la}\,y)
$$
for $x\ge y\ge0$ (recall that $r_\lambda^{N/D,+}(x,y)$ is symmetric in $x$ and $y$).

On the other hand, the function $\Phi_\sigma(u)=1/u^{\sigma}$,  $\Re(\sigma)>0$,  leads to the Riesz potential operator, and its kernel $R_\sigma$, if exists, is given by
\begin{equation}\label{riesz0}
R_\sigma(x,y) =\frac1{\Gamma(\sigma)}\int_0^\infty p_t(x,y) t^{\sigma-1}\,dt.
\end{equation}
It is easily seen that in the case when $0<\Re(\sigma)<d/2$, the above integral converges for $x,y\in\R$, $x\neq y$, and it is known that then,
$$
   R_\sigma(x,y) = c_{d,\sigma}\,\frac{1}{\|x-y\|^{d-2\sigma}},
$$
where $c_{d,\sigma}= {\Gamma(d/2-\sigma)}/({2^{2\sigma}\pi^{d/2}\Gamma(\sigma)})$.
Thus, in the same range of $\sigma$, by Theorem \ref{thm:main}, for any $x,y\in\R_+$, $x\neq y$, we get
$$%\begin{equation}\label{riesz}
 R^{N/D,+}_\sigma(x,y) = c_{d,\sigma}\,\left(\frac{1}{\|x-y\|^{d-2\sigma}}\pm\frac{1}{\|\tilde{x}-y\|^{d-2\sigma}}\right)\/.
$$

The case $\sigma=1$ is special and then the operator is customary called the Newtonian potential operator and its kernel, if exists, the Newtonian potential. Note that this is also the limiting case $\lambda=0$ for the resolvent operators, and hence, equivalently, the kernel is also known as Green's function and will be denoted by $G$. Thus, for $d\ge3$  Green's function corresponding to $-\Delta$ exists and we have $G=R_1$ and, by Theorem \ref{thm:main}, for $x,y\in\R_+$, $x\neq y$, we get
$$
G^{N/D,+}(x,y)=c_{d,1}\,\left(\frac{1}{\|x-y\|^{d-2}}\pm\frac{1}{\|\tilde{x}-y\|^{d-2}}\right)\/.
$$

It is interesting to stop by for a moment to clear the picture of Newtonian potentials  for $d=1$. The Newtonian potential on $\mathbb R$ does not exist (the  integral defining $R_1(x,y)$ in \eqref{riesz0} diverges for any $x,y\in \mathbb R$). However, the Newtonian potential for the Dirichlet Laplacian on the half-line $(0,\infty)$ does exist. It may be easily checked that the integral defining $R_1^{D,+}(x,y)$, analogous to that in \eqref{riesz0}, with $\sigma=1$ but with $p_t$ replaced by $p_t^{D,+}$, converges (due to a cancellation) for any $x,y\in \mathbb R_+$, $x\neq y$. Moreover, a calculation shows that $R_1^{D,+}(x,y)=\min\{x,y\}$. To complete the picture we mention that the Newtonian potential for the Neumann Laplacian on the half-line $(0,\infty)$ does not exist. This is because, this time, the integral defining $R_1^{N,+}(x,y)$, analogous to that in \eqref{riesz0}, with $\sigma=1$ but with $p_t$ replaced by $p_t^{N,+}$, diverges for any $x,y\in \mathbb R_+$.

\subsection{Subordinate killed and reflected Brownian motion}
One of the consequences of Theorem \ref{thm:main} is the reflection principle for subordinate killed/reflected Brownian motion. Let $W=(W_t)_{t\geq 0}$ be a $d$-dimensional Wiener process and denote by $S=(S_t)_{t\geq 0}$ an independent subordinator, i.e. an increasing (a.s.) L\'evy process, with Laplace exponent $\phi$. The function $\phi$ is a Bernstein function vanishing at zero and it has the following integral representation
\begin{equation*}
   \phi(\lambda) = b\lambda+\int_{(0,\infty)}(1-e^{-\lambda t})\mu(dt)\/,\quad \lambda\geq 0\/,
\end{equation*}
where $b\geq 0$ and $\mu$ stands for a Borel measure on $(0,\infty)$ such that $\int_{(0,\infty)}(1\wedge t)\mu(dt)<\infty$. The process $X_t=W_{S_t}$ is called a subordinate Brownian motion.  For a given open set $\Omega\subset \R$ we can consider the process $X$ killed upon exiting $\Omega$ and obtain $X^\Omega$, a killed subordinate Brownian motion, which has been intensively studied in recent years (see \cite{SV} and references therein). However, we can reverse the order of subordinating and killing, i.e. we can consider a killed Wiener process $W^\Omega$ subordinated by $S$. The process $Y_t = W^\Omega_{S_t}$ is called a subordinate killed Brownian motion and its infinitesimal generator is $-\phi(-\Delta_{D,\Omega})$. Note that the processes $X$ and $Y$ are different. However, the process $Y$ is very natural, useful and frequently applied in studying properties of $X$ (see \cite{KSV}). In the same way we can consider subordinated reflected Brownian motion. In both cases the densities of transition probabilities exists (since $g_t^{D,\Omega}(x,y)$ and $g_t^{N,\Omega}(x,y)$ exists) and consequently, by Theorem \ref{thm:main}, the reflection principles hold for transition probability densities of subordinate killed and reflected Brownian motions.

\subsection{Heat kernels}
It is interesting to recover \eqref{ker1} and \eqref{ker2} for several open sets by using formulas for Neumann and Dirichlet heat kernels given in terms of series.
%%%%%%%%%%%%%%%%%%%%%%%%%%%%%%%%%%%%%%%%%%%%
\subsubsection{Intervals}
The Dirichlet heat kernel for an interval $(a,b)$ is given by the following formula (see, for instance, \cite[p. 10]{D2} and \cite[p. 108]{D1})
\begin{eqnarray*}
p_t^{D, (a,b)}(x,y) = \frac{2}{b-a}\sum_{n=1}^\infty e^{-\frac{n^2\pi^2t}{(b-a)^2}}\sin\left(\frac{n\pi}{b-a}(x-a)\right)\sin\left(\frac{n\pi}{b-a}(y-a)\right)\/,
\end{eqnarray*}
for $ x,y\in (a,b)$ and $t>0$. Due to the translational invariance and scaling property it is enough to consider the interval $(-\pi,\pi)$, where 
\begin{eqnarray*}
 p_t^{D, (-\pi,\pi)}(x,y) = \frac{1}{\pi}\sum_{n=1}^\infty e^{-\frac{n^2t}{4}}\sin\left(\frac{n}{2}(x+\pi)\right)\sin\left(\frac{n}{2}(y+\pi)\right)\/.
\end{eqnarray*}
Using standard trigonometric formulas, for $t>0$ and $x,y\in(0,\pi)$ we obtain
\begin{align*}
  p_t^{D, (-\pi,\pi)}(x,y)-p_t^{D, (-\pi,\pi)}(-x,y) &= \frac{2}{\pi}\sum_{n=1}^\infty e^{-\frac{n^2t}{4}}\sin\left(\frac{n x}{2}\right)\cos\left(\frac{n\pi}{2}\right)\sin\left(\frac{n}{2} (y+\pi)\right)\\
	&= \frac{2}{\pi}\sum_{k=1}^\infty e^{-k^2t}\sin\left(kx\right)\cos\left({k\pi}\right)\sin\left(k (y+\pi)\right)\\
	&=\frac{2}{\pi}\sum_{k=1}^\infty e^{-k^2t}\sin\left(kx\right)\sin\left(ky\right)\\
	&=p_t^{D, (0,\pi)}(x,y).
\end{align*}

Analogously, the Neumann heat kernel for $(a,b)$ is given by
$$
p_t^{N, (a,b)}(x,y) = \frac{2}{b-a}\sum_{n=0}^\infty e^{-\frac{n^2\pi^2t}{(b-a)^2}}\cos\left(\frac{n\pi}{b-a}(x-a)\right)\cos\left(\frac{n\pi}{b-a}(y-a)\right)\/,
$$
and similar calculations lead to
\begin{equation*}
  p_t^{N, (-\pi,\pi)}(x,y)+p_t^{N, (-\pi,\pi)}(-x,y)=p_t^{N, (0,\pi)}(x,y).
\end{equation*}

%%%%%%%%%%%%%%%%%%%%%%%%%%%%%%%%%%%%%%%%%%%%
\subsubsection{Cones on the plane}
Given $\Phi\in(0,2\pi]$ let $\Omega_\Phi$ denote the open (infinite) cone
\begin{equation}
\label{eq:cone}
\Om_\Phi=\{x=\rho e^{i\theta}\in\mathbb{R}^2\colon 0<\rho<\infty,\quad 0<\theta<\Phi\}
\end{equation}
on the plane with vertex at the origin and aperture $\Phi$. By $p_t^{D,\Phi}$ and $p_t^{N,\Phi}$ we shall denote the Dirichlet and Neumann heat kernels related to $\Om_\Phi$, respectively. 

We shall verify the formula
\begin{equation}\label{ker4}
p_t^{D,\Phi/2}(x,y)=p_t^{D,\Phi}(x,y)-p_t^{D,\Phi}(\tilde x,y) ,\qquad x,y\in\Om_{\Phi/2},
\end{equation}
%$0<\Phi\le2\pi$, 
where $\tilde x=\rho e^{i(\Phi- \theta)}$ denotes the reflection of $x=\rho e^{i\theta}$ with respect to the bisector of the cone $\Om_{\Phi}$,
by using an old Carslaw and Jaeger formula that expresses $p_t^{D,\Phi}(x,y)$ by a convergent series; see \cite[p.379]{CJ}.
This formula was generalized by Ba\~nuelos and Smits to higher dimensions, cf. \cite[Lemma]{BS}. Specifying \cite[(2.2)]{BS} to dimension 2 
 one gets for $x=\rho e^{i\theta}\in \Om_{\Phi}$, $y=re^{i\eta}\in \Om_{\Phi}$ (with the rescaling $t\to 2t$, to stick to our setting), 
$$
p_t^{D,\Phi}(x,y)=\frac1{2\Phi t}\exp\left(-\frac{\rho^2+r^2}{4t}\right)\sum_{j=1}^\infty I_{\pi j/\Phi}\Big(\frac{\rho r}{2t}\Big) 2\sin\left(j\frac{\pi}{\Phi}\theta\right)\sin\left(j\frac{\pi}{\Phi}\eta\right),
$$
where $I_\nu(z)$ denotes the modified Bessel function of order $\nu$. Using the elementary formula for the product of sines leads to
\begin{equation}\label{aa}
p_t^{D,\Phi}(x,y)=\frac1{2\Phi t}\exp\left(-\frac{\rho^2+r^2}{4t}\right)\left(B^\Phi(\frac{\rho r}{2t},\theta-\eta)- B^\Phi(\frac{\rho r}{2t},\theta+\eta)\right),
\end{equation}
where
$$
B^\Phi(\tau,\gamma)=\sum_{j=1}^\infty I_{\pi j/\Phi}(\tau)\cos\left(j\frac\pi\Phi \gamma\right).
$$
Note also that 
$$
B^\Phi(\tau,\Phi-\gamma)=\sum_{j=1}^\infty I_{\pi j/\Phi}(\tau)\cos\left(\frac{j\pi}\Phi(\Phi-\gamma)\right)=\sum_{j=1}^\infty (-1)^jI_{\pi j/\Phi}(\tau)\cos\left(\frac{j\pi}\Phi\gamma\right)\/.
$$
Consequently, we have 
$$
B^\Phi(\tau,\gamma)+B^\Phi(\tau,\Phi-\gamma)=2 \sum_{j=1}^\infty I_{2\pi j/\Phi}(\tau)\cos\left(\frac{2j\pi}\Phi\gamma\right)=2B^{\Phi/2}(\tau,\gamma),
$$
and now \eqref{ker4} follows from \eqref{aa}.

The case $\Phi=2\pi$ deserves an additional comment. Then $\Om_{2\pi}$ is the whole plain with the non-negative $x_1$-axis $\{(x_1,0)\colon x_1\ge0\}$ removed, $\{p_t^{D,2\pi}\}$ is 
the Dirichlet heat kernel for this open set and we have 
\begin{equation}\label{ker5}
p_t^{D,\pi}(x,y)=p_t^{D,2\pi}(x,y)-p_t^{D,2\pi}(\tilde x,y) ,\qquad x,y\in\Om_{\pi}.
\end{equation}
However, the heat kernel for the upper half-plane $\Om_\pi=\mathbb{R}_+^2$, $\{p_t^{D,\pi}\}$, is also expressible through the Euclidean heat kernel $\{p_t\}$ on the plane and hence \eqref{ker5} 
holds with $p_t$ replacing $p_t^{D,\pi}$ on the right hand side of \eqref{ker5}. We check this by a direct calculation. Indeed,
$$
B^\pi(\tau,\gamma)=\frac12\left(e^{\tau\cos\gamma}-I_0(\tau)\right),
$$
(see \cite[9.6.34]{AS}), and hence for $x=(x_1,x_2)=\rho e^{i\theta}$, $y=(y_1,y_2)=re^{i\eta}$, $x_2,y_2>0$, 
we have
\begin{align*}
p_t^{D,\pi}(x,y)&=\frac1{2\pi t}\exp\left(-\frac{\rho^2+r^2}{4t}\right)\left(B^\pi(\frac{\rho r}{2t},\theta-\eta)- B^\pi(\frac{\rho r}{2t},\theta+\eta)\right)\\
&=\frac1{4\pi t}\exp\left(-\frac{\rho^2+r^2}{4t}\right)\left(\exp\left(\frac{\rho r}{2t}\cos(\theta-\eta)\right)-\exp\left(\frac{\rho r}{2t}\cos(\theta+\eta)\right)\right)\\
&=\frac1{4\pi t}\exp\left(-\frac{\rho^2+r^2}{4t}\right)\exp\left(\frac{\rho r}{2t}\cos\theta\cos\eta\right)\\
&\,\,\,\,\,\,\,\,\,\,\,\,\,\,\,\,\times\left(\exp\left(\frac{\rho r}{2t}\sin\theta\sin\eta\right)-\exp\left(-\frac{\rho r}{2t}\sin\theta\sin\eta\right)\right)\\
&=\frac1{4\pi t}\exp\left(-\frac{(x_1-y_1)^2}{4t}\right)\left(\exp\left(-\frac{(x_2-y_2)^2}{4t}\right)-\exp\left(-\frac{(x_2+y_2)^2}{4t}\right)\right)\\
&=p_t(x,y)-p_t(\tilde x,y).
\end{align*}

Although the case of the Neumann Laplacian on cones was not discussed in \cite{BS}, it is clear that repeating the arguments from the proof of \cite[Lemma 1]{BS} (specified to 
the dimension 2 and with the rescaling $t\to 2t$) leads to the following.
%%%%%%%%%%%%%%%%%%%%%%%%%%%%%%%%%%%%%%%%%%%
\begin{lemma} \label{lem:one}
The Neumann heat kernel related to the cone $\Om_\Phi$, $0<\Phi\le 2\pi$, is given, in polar coordinates, by
\begin{equation*}%\label{ker2}
p_t^{N,\Phi}(x,y)=\frac1{2\Phi t}\exp\left(-\frac{\rho^2+r^2}{4t}\right)\sum_{j=0}^\infty I_{\pi j/\Phi}\left(\frac{\rho r}{2t}\right) 2\cos\left(j\frac{\pi}{\Phi}\theta\right)\cos\left(j\frac{\pi}{\Phi}\eta\right).
\end{equation*}
\end{lemma}
%%%%%%%%%%%%%%%%%%%%%%%%%%%%%%%%%%%%%%%%%%%
Then the the product formula for the cosines leads to
\begin{equation*}%\label{bb}
p_t^{N,\Phi}(x,y)=\frac1{2\Phi t}\exp\left(-\frac{\rho^2+r^2}{4t}\right)\left(B^\Phi \left(\frac{\rho r}{2t},\theta-\eta\right)+ B^\Phi\left(\frac{\rho r}{2t},\theta+\eta\right)\right),
\end{equation*}
and, consequently, we obtain the following.
%%%%%%%%%%%%%%%%%%%%%%%%%%%%%%%%%%%%%%%%%%%%
\begin{corollary} \label{cor:two}
For $0<\Phi\le 2\pi$ we have
\begin{equation}\label{ker3}
p_t^{N,\Phi/2}(x,y)=p_t^{N,\Phi}(x,y)+p_t^{N,\Phi}(\tilde x,y) ,\qquad x,y\in\Om_{\Phi/2}.
\end{equation}
\end{corollary}
The remark concerning the limiting case $\Phi=2\pi$ also applies here. 

If $\Phi_n=2\pi/2^n$, $n\in\N_0$, then for the \textit{diadic} cones  $\Om_{\Phi_n}$, \eqref{ker4} gives the following recurrence relation for the kernels $p_t^{D,\Phi_n}$:
$$
p_t^{D,\Phi_{n+1}}(x,y)=p_t^{D,\Phi_n}(x,y)-p_t^{D,\Phi_n}(\s_nx,y), \qquad x,y\in \Om_{\Phi_{n+1}}.
$$
Here $\s_n$ denotes the reflection with respect to the bisector of the cone $\Om_{\Phi_n}$. 
In the case of the first quarter $\Om_{\pi/2}$, for $x=(x_1,x_2)$, $y=(y_1,y_2)$, $x_1,x_2,y_1,y_2>0$,
\begin{align*}
p_t^{D,\pi/2}(x,y)&=\frac1{4\pi t}\left(\exp\left(-\frac{(x_1-y_1)^2}{4t}\right) - \exp\left(-\frac{(x_1+y_1)^2}{4t}\right) \right)\\
&\,\,\,\,\,\,\,\,\times\left(\exp\big(-\frac{(x_2-y_2)^2}{4t}\big)-\exp\big(-\frac{(x_2+y_2)^2}{4t}\big)\right).
\end{align*}
In the case of the cone $\Om_{\pi/4}$ with aperture $\pi/4$, for $x=(x_1,x_2)$ and $y=(y_1,y_2)$, $0<x_2<x_1,0<y_2<y_1$,
\begin{align*}
p_t^{D,\pi/4}(x,y)=&\frac1{4\pi t}\Big(\exp\big(-\frac{(x_1-y_1)^2}{4t}\big) - \exp\big(-\frac{(x_1+y_1)^2}{4t}\big) \Big)\\
&\,\,\,\,\,\,\,\,\times
\Big(\exp\big(-\frac{(x_2-y_2)^2}{4t}\big)-\exp\big(-\frac{(x_2+y_2)^2}{4t}\big)\Big)\\
&-\frac1{4\pi t}\Big(\exp\big(-\frac{(x_2-y_1)^2}{4t}\big) - \exp\big(-\frac{(x_2+y_1)^2}{4t}\big) \Big)\\
&\,\,\,\,\,\,\,\,\times\Big(\exp\big(-\frac{(x_1-y_2)^2}{4t}\big)-\exp\big(-\frac{(x_1+y_2)^2}{4t}\big)\Big).
\end{align*}
Analogous comments are in order for the Neumann heat kernel associated with the diadic cones.
%%%%%%%%%%%%%%%%%%%%%%%%%%%%%%%%%%%%%%%%%%%%

\subsubsection{Truncated cones}
Given $\Phi\in(0,2\pi]$ let $\widehat\Omega_\Phi$ denote 
the open truncated cone
$$
\widehat\Om_\Phi=\{\rho e^{i\theta}\in\mathbb{R}^2\colon 0<\rho<1,\quad 0<\theta<\Phi\},
$$
with vertex at the origin, aperture $\Phi$, and 'radius' $R=1$ (restricting the attention to $R=1$ does not limit the generality). Accordingly, by $\widehat{p_t}^{D,\Phi}$ and 
$\widehat{p_t}^{N,\Phi}$ we shall denote the Dirichlet and Neumann heat kernels related to $\widehat\Om_\Phi$. Given $\nu>-1$ let $\{b_{\nu,s}\}_{s=1}^\infty$ denote the 
increasing sequence of the consecutive positive zeros of the Bessel function $J_\nu$. 

The functions 
$$
\widehat\v^\Phi_{j,s}(x)=\frac2{\sqrt{\Phi}}d_{j,s}(\Phi)J_{\pi j/\Phi}(b_{\pi j/\Phi,s}\rho)\sin(j\frac\pi\Phi\theta),\qquad j,s\in\N,
$$
where $x=\rho e^{i\theta}$ and
$$
d_{j,s}(\Phi)=\frac1{J_{1+\pi j/\Phi}(b_{\pi j/\Phi,s})},
$$ 
constitute the complete orthonormal system in $L^2(\widehat\Om_\Phi,dx)$. Moreover, $\widehat\v^\Phi_{j,s}$ are eigenfunctions of $-\Delta_{D,\widehat\Om_\Phi}$, that is 
$\widehat\v^\Phi_{j,s}\in{\D}(-\Delta_{D,\widehat\Om_\Phi})$, and
$$
-\Delta_{D,\widehat\Om_\Phi}=\lambda^\Phi_{j,s}\widehat\v^\Phi_{j,s},\qquad \lambda^\Phi_{j,s}=(b_{\pi j/\Phi,s})^2.
$$
Hence, for $x=\rho e^{i\theta}, y=r e^{i\eta}\in \widehat\Om_\Phi$ and  $t>0$,
$$
\widehat{p_t}^{D,\Phi}(x,y)=\frac4\Phi\sum_{j,s=1}^\infty d_{j,s}(\Phi)^2e^{-\lambda^\Phi_{j,s}t}\widehat\v^\Phi_{j,s}(\rho,\theta)\widehat\v^\Phi_{j,s}(r,\eta),
$$
 and the product formula for the sines leads to
$$
\widehat{p_t}^{D,\Phi}(x,y)=\frac2\Phi\Big(\hat B^\Phi(t,\rho,r;\theta-\eta)- \hat B^\Phi(t,\rho,r;\theta+\eta)\Big),
$$
where
$$
\hat{B}^\Phi(t,\rho,r;\gamma)=\sum_{j,s=1}^\infty d_{j,s}(\Phi)^2e^{-\lambda^\Phi_{j,s}t}J_{\pi j/\Phi}(b_{\pi j/\Phi,s}\,\rho)J_{\pi j/\Phi}(b_{\pi j/\Phi,s}\,r)\cos(j\frac\pi\Phi\gamma).
$$
Note that
$$
\hat{B}^\Phi(t,\rho,r;\Phi-\gamma)=\sum_{j=1}^\infty (-1)^j \sum_{s=1}^\infty d_{j,s}(\Phi)^2e^{-\lambda^\Phi_{j,s}t}J_{\pi j/\Phi}(b_{\pi j/\Phi,s}\rho)J_{\pi j/\Phi}(b_{\pi j/\Phi,s}r)\cos(j\frac\pi\Phi\gamma),
$$
and, consequently (note that $d_{j,s}(\Phi/2)=d_{2j,s}(\Phi)$ and $\lambda^{\Phi/2}_{j,s}=\lambda^\Phi_{2j,s}$), we have 
$$
\hat{B}^\Phi(t,\rho,r;\gamma)+\hat{B}^\Phi(t,\rho,r;\Phi-\gamma)=2\hat{B}^{\Phi/2}(t,\rho,r;\gamma).
$$
This leads to
\begin{equation}\label{trcone}
\widehat{p_t}^{D,\Phi/2}(x,y)=\widehat{p_t}^{D,\Phi}(x,y)-\widehat{p_t}^{D,\Phi}(\tilde x,y) ,\qquad x,y\in\widehat\Om_{\Phi/2}.
\end{equation}

Specified to $\Phi=2\pi$, \eqref{trcone} represents the reflection formula for the truncated cone $\widehat\Om_{2\pi}$, i.e. the open unit ball $B=\{(x_1,x_2)\colon x_1^2+x_2^2<1\}$ with
the segment $\{(x_1,0) \colon 0\le x_1\le1\}$ removed, where $\{\widehat{p_t}^{D,2\pi}\}$ is the Dirichlet heat kernel on $\widehat\Om_{2\pi}$. Clearly, the analogous formula  holds with the left-hand side unchanged but with $\{\widehat{p_t}^{D,B}\}$, the Dirichlet heat kernel on $B$, replacing $\{\widehat{p_t}^{D,2\pi}\}$ on the right-hand side of \eqref{trcone}. 
We shall check that this is indeed the case by writing down explicitly  $\widehat{p_t}^{D,B}$ and comparing it with $\widehat{p_t}^{D,\pi}$. 

It is known, see for instance \cite[Theorem 5.4, p.\,151]{F}, that the system of functions
$$
\Big\{\frac{\sqrt2}{\sqrt\pi J_{j+1}(b_{j,s})}J_{j}(b_{j,s}\rho)\left\{\sin,\cos\right\}
 (j\theta)\colon j,s\in\N\Big\}\cup\Big\{\frac1{\sqrt{\pi}J_{1}(b_{0,s})}J_{0}(b_{0,s}\rho) \colon s\in\N\Big\},
$$
is an ortonormal basis in $L^2(B, dx)$, and, moreover, the functions are eigenfunctions of $-\Delta_{D,B}$ corresponding to the eigenvalues $b_{j,s}^2$, $j,s\in\N$. 
Hence, the Dirichlet heat kernel for $B$ is given by 
\begin{align*}
\widehat{p_t}^{D,B}(x,y)&=\frac2\pi\sum_{j=1}^\infty\sum_{s=1}^\infty \frac1{J_{j+1}(b_{j,s})^2}e^{-b^2_{j,s}t}J_{j}(b_{j,s}\rho)J_{j}(b_{j,s}r)\big(\cos j\theta \cos j\eta+\sin j\theta \sin j\eta\big)\\
&+\frac1\pi\sum_{s=1}^\infty\frac1{J_{1}(b_{0,s})^2}e^{-b^2_{0,s}t}J_{0}(b_{0,s}\rho)J_{0}(b_{0,s}r),
\end{align*}
for $t>0$, $x=\rho e^{i\theta}$, $y=r e^{i\eta}$, $0<\rho,r<1$, $0<\theta,\eta\le2\pi$. Recall that
$$
\widehat{p_t}^{D,\pi}(x,y)=\frac2\pi\sum_{j=1}^\infty\sum_{s=1}^\infty 
\frac1{J_{j+1}(b_{j,s})^2}e^{-b^2_{j,s}t}J_{j}(b_{j,s}\rho)J_{j}(b_{j,s}r)\big(\cos j(\theta-\eta)-\cos j(\theta+\eta)\big).
$$
It is now easily seen that the identity
$$
\widehat{p_t}^{D,\pi}(x,y)=\widehat{p_t}^{D,B}(x,y)- \widehat{p_t}^{D,B}(\tilde x,y)
$$
follows for $x=\rho e^{i\theta}$, $y=re^{i\eta}$, $0<\theta,\eta<\pi$, where, as before,  $\tilde x=\rho e^{i(2\pi-\theta)}$.

Finally, note that using the orthonormal basis $\{\widehat\psi^\Phi_{j,s}\}_{j,s=1}^\infty$ that consists of eigenfunctions of $-\Delta_{N,\widehat\Om_\Phi}$, where $\widehat\psi^\Phi_{j,s}$ differs from $\widehat\v^\Phi_{j,s}$ by replacing the sines by the cosines, and applying the  arguments analogous to these just used, gives the corresponding formula for the Neumann heat kernel, that is
\begin{equation*}%\label{ker4}
\widehat{p_t}^{N,\Phi/2}(x,y)=\widehat{p_t}^{N,\Phi}(x,y)+\widehat{p_t}^{N,\Phi}(\tilde x,y) ,\qquad x,y\in\widehat\Om_{\Phi/2}.
\end{equation*}
The comments concerning the case $\Phi=2\pi$ apply here as well.

%%%%%%%%%%%%%%%%%%%%%%%%%%%%%%%%%%%%%%%%%%%%%%%


\begin{thebibliography}{99}

\bibitem{AS} M. Abramowitz, I.A. Stegun, Handbook of Mathematical Functions with Formulas, Graphs,
and Mathematical Tables, 9th edn. Dover, New York, 1972.

\bibitem{BS} R. Ba\~nuelos, R.G. Smits, Brownian motion in cones, Probab. Theory Related Fields 108 (1997), 299-319.

\bibitem{Br} H. Brezis, Functional Analysis, Sobolev Spaces and Partial Differential Equations, Universitext, Springer, 2011.

\bibitem{CJ} H.S. Carslaw, J.C. Jaeger, Conduction of Heat in Solids,  Oxford University Press, 1959.

\bibitem{Co} J.B. Conway, A course in functional analysis, Graduate Texts in Mathematics 96, Springer, 1990.

\bibitem{D1} E.B. Davies, Heat Kernels and Spectral Theory, Cambridge Tracts in Mathematics 92, Cambridge University Press, 1989. 

\bibitem{D2} E.B. Davies, Spectral Theory and Differential Operators, Cambridge studies in advanced mathematics 42,  Cambridge University Press, 1994. 

\bibitem{F} G.B. Folland, Fourier Analysis and its applications,  Wadsworth \& Brooks\slash Cole Advanced Books and Software, Pacific Grove, 1992.

\bibitem{GR} I.S.~Gradstein, I.M.~Ryzhik, Table of Integrals, Series and Products, 7th edn. Academic Press, New York, 2007.

\bibitem{Hu} J. E. Humpreys, Reflection groups and Coxeter groups, Cambridge University Press, 1990.

\bibitem{KSV} P. Kim, R. Song, Z. Vondracek, Potential theory of subordinate killed Brownian motion, Trans. Amer. Math. Soc. (2019), doi.org/10.1090/tran/7358. 

\bibitem{RS} M. Reed, B. Simon, Methods of Modern Mathematical Physics II. Fourier Analysis and Self-Adjointness, Academic Press, New York, 1975. 

\bibitem{Sch} K. Schm\"udgen, Unbounded Self-adjoint Operators on Hilbert Space, Graduate Texts in Mathematics 265, Springer, 2012.

\bibitem{SV} R. Song, Z. Vondracek, Potential Theory of Subordinate Brownian Motion. In: 
Potential Analysis of Stable Processes and its Extensions. Lecture Notes in Mathematics, vol. 1980, Springer, 2009.

\end{thebibliography}
\end{document}